\documentclass[10pt]{article}
\usepackage{amssymb}
\usepackage{amsmath}
\usepackage{theorem}
\usepackage{epsfig}
\usepackage{verbatim}
\usepackage{graphicx}
\usepackage{enumerate}
\usepackage{enumitem}
\usepackage{cancel}
\usepackage{mathtools}
\usepackage{hyperref}
\usepackage{datetime}

\usepackage[square,sort,comma,numbers]{natbib}
\usepackage{color}

\newlength{\hchng}
\newlength{\vchng}

\newcommand {\bc} {\begin{center}}
\newcommand {\ec} {\end{center}}

\theoremstyle{plain}

\newtheorem{thm}{Theorem}[section]
\newtheorem{lem}[thm]{Lemma}
\newtheorem{defn}[thm]{Definition}
\newtheorem{prop}[thm]{Proposition}
\newtheorem{ques}[thm]{Question}
	
\newtheorem{rem}{Remark}[section]

\newtheorem{conj}{Conjecture}

\allowdisplaybreaks[2]

\newenvironment{proof}[1]{\begin{trivlist} \item[] {\em Proof of #1:}}{\hfill $\Box$
                      \end{trivlist}}

\newcommand{\ud}{\,\mathrm{d}}
\newcommand{\la}{\lambda}

\newcommand{\R}{\mathbb{R}}

\newcommand{\pa}{\partial}

\newcommand{\norm}[1]{\left\lVert#1\right\rVert}

\newcommand{\Addresses}{{% additional braces for segregating \footnotesize
  \bigskip
  \footnotesize

  T.~Beck, \textsc{Department of Mathematics, University of North Carolina,
    Chapel Hill, North Carolina}\par\nopagebreak
  \textit{E-mail address}: \texttt{tdbeck@email.unc.edu}

}}

\title{Localization of the first eigenfunction of a convex domain}

\author{Thomas Beck}

\date{\today}                                           % Activate to display a given date or no date

\begin{document}
\maketitle
\begin{abstract}
We study the first Dirichlet eigenfunction of the Laplacian in a $n$-dimensional convex domain. For domains of a fixed inner radius, estimates of Chiti \cite{Ch1}, \cite{Ch2}, imply that the ratio of the $L^2$-norm and $L^{\infty}$-norm of the eigenfunction is minimized when the domain is a ball. However, when the eccentricity of the domain is large the eigenfunction should spread out at a certain scale and this ratio should increase. We make this precise by obtaining a lower bound on the $L^2$-norm of the eigenfunction  and show that the eigenfunction cannot localize to too small a subset of the domain. As a consequence, we settle a conjecture of van den Berg, \cite{vdB}, in the general $n$-dimensional case. The main feature of the proof is to obtain sufficiently sharp estimates on the first eigenvalue in order to estimate the first  derivatives of the eigenfunction.

\end{abstract}
\section{Introduction and statement of results}

Let $\Omega\subset \R^n$ be a convex domain and let $\lambda$ be the first eigenvalue of the Dirichlet Laplacian on $\Omega$. We denote the corresponding eigenfunction by $u$ so that
$$
\begin{cases}
(\Delta+\lambda)u  =  0 \text{ in } \Omega \\
\phantom{(\Delta+\lambda)}u  =  0 \text{ on } \pa\Omega. 
\end{cases}
$$
This first eigenfunction is of one sign, and we choose it so that $u(x)>0$ in $\Omega$. Our starting point for studying the behaviour of $u$ and its level sets is that the convexity of $\Omega$ ensures that $u$ is log-concave, \cite{BL}. In particular the superlevel sets
\begin{align*}
\{x\in\Omega:u(x)>c\}
\end{align*}
are convex subsets of $\Omega$. It is natural to study the shape of the level sets of $u$ and how they depend on the geometry of $\Omega$ and the level under consideration. The quantity $|u(x)|^2$ can be interpreted as an (unnormalized) density for a free quantum particle in the domain $\Omega$. The shape and location of the superlevel sets where $u$ is comparable to its maximum value therefore correspond to the parts of $\Omega$ where the particle is most likely to be found. In this paper, we will obtain a lower bound on the $L^2(\Omega)$-norm of $u$ in terms of its $L^{\infty}(\Omega)$-norm and length scales coming from the shape of $\Omega$ (see Theorem \ref{thm:Main} below). %This result can be interpreted as showing that the first eigenfunction $u$ cannot localize to too small a subset of $\Omega$. 
 Where Laplace eigenfunctions localize, that is the region of $\Omega$ where they are of large magnitude relative to the rest of the domain, has received recent attention. For example, the torsion function has been used as a landscape function for predicting where Laplace eigenfunctions will localize, \cite{FM}, \cite{ADJMF}, \cite{SS1}. While in general the first eigenfunction can localize to a small subset of $\Omega$, relative to $\Omega$ itself, our result will place a restriction on how small this region can be.
%**sentence / re-frame using the word localization and maybe in the abstract..eigenfunctions can in general localize to a small region..explain what this means..we show how small a region**

In \cite{Ch1}, \cite{Ch2}, Chiti provides a lower bound on the $L^{2}(\Omega)$-norm of $u$ of the form
\begin{align} \label{eqn:Chiti}
\norm{u}_{L^{2}(\Omega)} \geq c_n^*\text{ inrad}(\Omega)^{n/2}\norm{u}_{L^{\infty}(\Omega)}
\end{align}
Here inrad$(\Omega)$ is the inner radius of $\Omega$.  The constant $c_n^*>0$ depends only on the dimension, and is explicitly given in terms of Bessel functions (and their zeros). (In fact, this bound  holds for any bounded, connected domain $\Omega$.) Moreover, the constant $c_n^*$ cannot be improved since equality in \eqref{eqn:Chiti} holds when $\Omega$ is a ball. However, for $\Omega$ convex and when the diameter of $\Omega$ is large compared to its inner radius, one expects the eigenfunction to \textit{spread out} along the diameter of $\Omega$, and for the $L^{2}(\Omega)$-norm to increase relative to the $L^{\infty}(\Omega)$-norm. In terms of the estimate in \eqref{eqn:Chiti}, the question is then whether an estimate of the form
\begin{align} \label{eqn:Chiti2}
\norm{u}_{L^{2}(\Omega)} \geq c_n \left(\text{diam}(\Omega)/\text{inrad}(\Omega)\right)^{\alpha}\text{inrad}(\Omega)^{n/2}\norm{u}_{L^{\infty}(\Omega)}
\end{align}
holds for all convex $\Omega$, and some uniform $\alpha>0$. Repeated applications of the Harnack inequality in overlapping balls is not sufficient to establish \eqref{eqn:Chiti2} for any $\alpha>0$, and so any improvement of \eqref{eqn:Chiti} must use the fact that $u$ is an eigenfunction in a fundamental way. Kr\"oger, \cite{Kr1}, in two dimensions, and van den Berg, \cite{vdB}, in higher dimensions studied the first eigenfunction of a thin sector. Via a separation of variables in polar coordinates, and the properties of the resulting Bessel function in the radial variable, this example of the sector ensures that the maximal value of $\alpha$ for which \eqref{eqn:Chiti2} could hold is $\alpha = \tfrac{1}{6}$. Based on the intuition that the sector should be the convex domain for which the eigenfunction spreads out the least, van den Berg made the following conjecture:
\begin{conj}[\cite{vdB}] \label{conj:vdB}
There exists a constant $c_n>0$, depending only on the dimension $n$, such that 
\begin{align*}
\norm{u}_{L^{2}(\Omega)} \geq c_n \left(\emph{diam}(\Omega)/\emph{inrad}(\Omega)\right)^{1/6}\emph{inrad}(\Omega)^{n/2}\norm{u}_{L^{\infty}(\Omega)}.
\end{align*}
\end{conj}
The two dimensional case of this conjecture has been established in \cite{GMS}. Their proof uses an eigenvalue bound for the first eigenvalue of a class of one dimensional Schr\"odinger operators, and the work of Grieser and Jerison, \cite{Je1}, \cite{GJ1}, on the first eigenfunction of a convex, planar domain.

In this paper, we bound $\norm{u}_{L^2(\Omega)}$ from below in the general $n$-dimensional case. We call $K$ a John ellipsoid associated to $\Omega\subset\R^n$ if $K$ is contained within $\Omega$ and the dilation of $K$ about its centre with scaling factor $n$ contains $\Omega$. John's lemma \cite{Jo} ensures that such an ellipsoid $K$ exists. We now fix a John ellipsoid $K$ and define $N_j$ to be the lengths of the axes of $K$ with
\begin{align*}
N_1\geq N_2\geq\cdots \geq N_n.
\end{align*}
Our main theorem provides a lower bound on the scale at which the eigenfunction can localize by establishing a lower bound on the $L^{2}(\Omega)$-norm of $u$ in terms of its $L^{\infty}(\Omega)$-norm, and the length scales $N_j$. 
\begin{thm} \label{thm:Main}
There exists a constant $c_n>0$, depending only on the dimension $n$, such that
\begin{align*}
\norm{u}_{L^2(\Omega)} \geq c_n N_{n}^{n/2}\prod_{j=1}^{n-1}\left(N_j/N_n\right)^{1/6} \norm{u}_{L^{\infty}(\Omega)}. 
\end{align*}
\end{thm}
In particular, $\prod_{j=1}^{n-1}\left(N_j/N_n\right)^{1/6}\geq (N_1/N_n)^{1/6}$, and $N_1$, $N_n$ are comparable to the diameter, inner radius of $\Omega$, up to a factor depending only on $n$. Therefore, Theorem \ref{thm:Main} settles Conjecture \ref{conj:vdB}.
\begin{rem} \label{rem:Main}
Let $M_1\geq M_2\geq\cdots \geq M_n$ be the lengths of the axes of a John ellipsoid for the superlevel set $\{x\in\Omega:u(x)>\tfrac{1}{2}\max_{\Omega}u\}$. In the course of proving Theorem \ref{thm:Main} we will show that $M_j\geq c_nN_j^{1/3}$ for some constant $c_n>0$. In terms of localization, this shows that the eigenfunction does not localize in a subset of $\Omega$ smaller than this.
\end{rem}
To prove Theorem \ref{thm:Main} we first obtain an upper bound on the directional derivatives of $u$ in terms of the length scales $N_j$. After a rotation we will assume that the axes of $K$ lie along the coordinate axes. 
\begin{thm} \label{thm:Main1}
There exists a constant $C_n$, depending only on the dimension $n$, such that for each $j$, $1\leq j \leq n$, the derivative $\pa_{x_j}u(x)$ satisfies
\begin{align*}
\norm{\pa_{x_j}u}_{L^2(\Omega)} \leq C_n N_n^{-1}\left(N_j/N_n\right)^{-1/3} \norm{u}_{L^2(\Omega)}.
%\int_{\Omega} \left|\pa_{x_j}u(x)\right|^2 \ud x \leq C_nN_1^{-2}\left(N_j/N_1\right)^{-2/3} \int_{\Omega}\left|u(x)\right|^2 \ud x.
\end{align*}
\end{thm}
\begin{rem} \label{rem:Main1}
If we denote $u_m$ to be the $m$-th Dirichlet eigenfunction of $\Omega$, then the estimate in Theorem \ref{thm:Main1} continues to hold, with a constant $C_n$ replaced by a constant $C_{m,n}$ depending only on $m$ and $n$.
\end{rem}
Via a dilation we can also assume that $N_n=1$ when proving these theorems, and by taking a constant multiple of $u$, we also assume that $\max_{\Omega}u = 1$. To prove Theorem \ref{thm:Main1} we will begin by using the eigenfunction equation to write
\begin{align*}
\int_{\Omega}\left|\nabla u\right|^2 \ud x = \la \int_{\Omega}|u|^2\ud x,
\end{align*}
and we will also use the variational formulation of the first eigenvalue,
\begin{align*}
\la = \inf\left\{ \frac{\int_{\Omega}\left|\nabla v\right|^2\ud x}{\int_{\Omega}|v|^2 \ud x}: v\in H_0^1(\Omega), v\neq0\right\}.
\end{align*}
These can be combined to reduce the proof of  Theorem \ref{thm:Main1} to obtaining sufficiently sharp upper bounds on the eigenvalue $\la$ in terms of the eigenvalues of $(n-j)$-dimensional cross-sections of $\Omega$ (see Proposition \ref{prop:derivative}). We prove the desired eigenvalue bounds by induction on $j$, and will carry out the proof in Section \ref{sec:gradient}. To prove Theorem \ref{thm:Main}, we will also use in a crucial way the log concavity of the eigenfunction $u$, \cite{BL}. In particular, this will allow us to reduce estimating the $L^2(\Omega)$-norm of $u$ to estimating the lengths of the axes of a John ellipsoid associated to the superlevel set
\begin{align*}
\Omega_{1/2} = \left\{x\in\Omega: u(x)>\tfrac{1}{2}\right\}.
\end{align*}
The desired estimate follows from using the derivative bounds in Theorem \ref{thm:Main1}, and we will prove Theorem \ref{thm:Main} in Section \ref{sec:L2}. Finally, in Section \ref{sec:2d} we discuss known estimates in the two dimensional case, and future directions in higher dimensions. In \cite{Je1}, Jerison introduces a length scale $L$ depending on the geometry of the convex, planar domain, and together with Grieser uses it to study the shape of the first (and second) eigenfunction, \cite{GJ1}, \cite{GJ2}. In particular, their results imply comparable upper and lower bounds on $\norm{u}_{L^2(\Omega)}$ in terms of this length scale $L$. It is natural to ask how to construct analogous length scales controlling the shape of the first eigenfunction in higher dimensions, and in Section \ref{sec:2d} we discuss this in more detail.
\begin{rem} \label{rem:constant}
Throughout, constants which we will denote by $C,C_1,c_1$ etc, are constants which depend only on the dimension. We also say that two quantities are comparable (and write as $\sim$) if they can be bounded in terms of each other up to a constant depending only on $n$.
\end{rem}

\section{Gradient bounds for the eigenfunction} \label{sec:gradient}

In this section we prove Theorem \ref{thm:Main1}. The key step in the proof is to obtain appropriate upper bounds on the eigenvalue $\lambda$. In fact, we will carry out an inductive step, which will require estimates on the first Dirichlet eigenvalue of $(n-k)$-dimensional cross-sections of $\Omega$ for $0\leq k \leq n-1$. To write down the eigenvalue bounds that we will establish, we first introduce the following notation:  Given $i$, with $1\leq i \leq n-k-1$, and a point $x\in\R^{n-k}$, we write
\begin{align*}
x = (x_1,x_2,\ldots,x_{n-k}) = (X_i,X_{n-k-i}') \in\R^{n-k},
\end{align*}
with $X_i\in \R^i$, $X_{n-k-i}'\in\R^{n-k-i}$. Now let $W$ be a $(n-k)$-dimensional convex domain. For each $Y_{i}\in\R^{i}$, we denote the $(n-k-i)$-dimensional cross-sections of $W$ by 
\begin{align*}
W(Y_{i}) = \left\{x=(X_i,X_{n-k-i}')\in W: X_i = Y_i\right\}.
\end{align*}
For us, $W$ will either be the original convex domain $\Omega$ (with $k=0$) or a $(n-k)$-dimensional cross-section of $\Omega$, for some $1\leq k \leq n-1$.
\begin{defn} \label{def:W-eigenvalue}
For a $(n-k)$-dimensional convex domain $W$, let $\lambda(W)$ be its first Dirichlet eigenvalue. For $i$, with $1\leq i \leq n-k-1$, and $Y_i\in\mathbb{R}^i$, let $\mu(Y_i;W)$ be the first Dirichlet eigenvalue of $W(Y_i)$, and define $\mu_i^*(W)$ by
\begin{align} \label{eqn:muj}
\mu_i^*(W) = \min_{Y_i}\mu(Y_i;W).
\end{align}
We also formally define $\mu_{n-k}^*(W) = 0$, and then for $1\leq i \leq n-k$ set
\begin{align*}
\delta_i(W) = \lambda(W) - \mu_i^*(W).
\end{align*}
\end{defn}
We can obtain gradient bounds on the first Dirichlet eigenfunction of $W$ in terms of $\delta_i(W)$ via the following proposition.
\begin{prop} \label{prop:derivative}
Let $u_W(x)$ be the first Dirichlet eigenfunction of $W$. Then, for each $1\leq i \leq n-k$, with $\delta_i(W)$ as in Definition \ref{def:W-eigenvalue}, the gradient bounds
\begin{align*}
\sum_{\ell = 1}^{i}\int_{W}\left|\pa_{x_\ell}u_W(x)\right|^2 \ud x \leq \delta_i(W)\int_{W}\left|u_W(x)\right|^2\ud x
\end{align*}
hold. In particular, $\delta_i(W)\geq0$ for all $i$.
\end{prop}
\begin{proof}{Proposition \ref{prop:derivative}}
Since $u_W$ is a Dirichlet eigenfunction with eigenvalue $\la(W)$ we have
\begin{align} \label{eqn:deriv1}
\int_{W}\left|\nabla u_W(x)\right|^2 \ud x = \la(W) \int_{W}|u_W(x)|^2 \ud x. 
\end{align}
For $i=n-k$, we have $\delta_{n-k}(W) = \lambda(W)$ and then the estimate holds (with equality) immediately. We now fix $i$ with $1\leq i <n-k$. For each $X_i\in\R^i$ such that $W(X_i)$ is non-empty, the function $u_W(X_i,\cdot)$ is an admissible test function for the first eigenvalue on $W(X_i)$. Therefore,
\begin{align*}
\sum_{\ell=i+1}^{n-k}\int_{W(X_i)}\left|\pa_{x_\ell}u_W(X_i,X_{n-k-i}')\right|^2 \ud X'_{n-k-i} & \geq \mu(X_i;W)\int_{W(X_i)}\left|u_W(X_i,X'_{n-k-i})\right|^2 \ud X'_{n-k-i} \\
& \geq \mu_i^*(W) \int_{W(X_i)}\left|u_W(X_i,X_{n-k-i}')\right|^2 \ud X'_{n-k-i} .
\end{align*}
Since this holds for each $X_i$, we integrate in $X_i$ and then use it in \eqref{eqn:deriv1} to get
\begin{align*}
 \mu_i^*(W) \int_{W}\left|u_W(x)\right|^2 \ud x + \sum_{\ell=1}^{i}\int_{W}\left|\pa_{x_\ell}u_W(x)\right|^2 \ud x \leq  \lambda(W) \int_{W}|u_W(x)|^2 \ud x.
\end{align*}
The estimate in the proposition then follows from the definition of $\delta_i(W)$.
\end{proof}
As before, we set $u(x) = u_{\Omega}(x)$, $\lambda = \lambda(\Omega)$, and for ease of notation, we  write $\mu(Y_i;\Omega) = \mu(Y_i)$, $\mu_i^* = \mu_i^*(\Omega)$. Using Proposition \ref{prop:derivative}, in order to prove Theorem \ref{thm:Main1} it is sufficient to establish the following eigenvalue bounds.
\begin{prop} \label{prop:eigenvalue}
There exists a constant $C_n$ such that for all $j$, $1\leq j \leq n$,
\begin{align*}
\mu_j^* \leq \lambda \leq \mu_j^* + C_nN_j^{-2/3}.
\end{align*}
\end{prop}
From Proposition \ref{prop:derivative} we have $\la-\mu_j^*\geq0$, and so we only need to prove the upper bound. Since $\mu_n^* = 0$, and $\Omega$ has inner radius comparable to $N_n= 1$, the estimate in the proposition certainly holds for $j=n$. We will prove Proposition \ref{prop:eigenvalue} by induction on $j$ (starting with $j=n$ as the base case, and then decreasing $j$). To establish the inductive step we will use the variational formulation of the first Dirichlet eigenvalue. We will construct an appropriate test function involving the eigenfunctions corresponding to the minimal eigenvalue $\mu_j^*$ of the $j$-dimensional cross-sections of $\Omega$. To demonstrate the method let us first use it to prove the proposition in the two dimensional case. (In two dimensions, the estimate in Proposition \ref{prop:eigenvalue} is also contained in the work of Jerison \cite{Je1} and Grieser-Jerison \cite{GJ1}.)
\begin{proof}{Proposition \ref{prop:eigenvalue} in two dimensions}
In the two dimensional case, we just need to consider $j=1$. After a translation along the $x_1$-axis, we may assume that the minimal value $\mu_1^* = \mu_1(Y_1)$ is attained at $Y_1=0$. (Note that this point is at a point where the height of the domain $\Omega$ in the $x_2$-direction is largest.) Let $\psi(x_2)$ be the corresponding $L^2(\Omega(0))$-normalized first Dirichlet eigenfunction of the interval $\Omega(0)$, extended to be zero outside of $\Omega(0)$.  By the properties of the John ellipsoid of $\Omega$, we can find a point $x=(x_1,x_2)\in\Omega$ with $|x_1| = N_1$, and so without loss of generality, we assume that $x^* = (N_1,x_2^*)\in\Omega$ for some $x_2^*$. By translating in the $x_2$-direction we may assume that $x_2^*=0$, and after this translation there still exists a constant $C$ such that $|x_2|\leq C$ on the support of $\psi(x_2)$.
\\
\\
We now define a test function that we can use in the variational formulation of the first eigenvalue $\lambda$: We set $v(x_1,x_2)$ to be the function
\begin{align} \label{eqn:eigenvalue1}
v(x_1,x_2) = \chi(x_1) \psi\left(x_2 N_1/(N_1-x_1)\right).
\end{align}
Here $\chi(x_1)\geq0$ is a smooth cut-off function, such that
\begin{align*}
\chi(x_1) & = 1 \text{ for } \tfrac{1}{2}N_1^{1/3} \leq x_1\leq N_1^{1/3}, \\
\chi(x_1) & = 0 \text{ for } x_1\geq 2N_1^{1/3}, \text{ } x_1\leq \tfrac{1}{4}N_1^{1/3}.
\end{align*} 
The function $\chi(x_1)$ can in particular be chosen so that $|\chi'(x_1)|\leq CN_1^{-1/3}$. The domain $\Omega$ contains the interval $\Omega(0)$ and the point $x^* = (N_1,0)$, and so also contains the convex hull of these two sets. Therefore, for each $x_1\in[0,N_1]$, the cross-section $\Omega(x_1)$ contains the interval $\tfrac{N_1-x_1}{N_1}\Omega(0)$.  
%If the interval $\Omega(0)$ is given by $[a,b]$, then for fixed $x_1\in[0,N_1]$, the function $\psi\left(x_2 N_1/(N_1-x_1)\right)$ is non-zero for
%\begin{align*}
%\frac{(N_1-x_1)}{N_1}a \leq x_2 \leq \frac{(N_1-x_1)}{N_1}b.
%\end{align*}
In particular, this ensures that $v(x_1,x_2)$ is equal to zero on the complement of $\Omega$, and we can use it in the variational formulation of the first eigenvalue $\la$. That is,
\begin{align} \label{eqn:eigenvalue2}
\la \leq \frac{\int_{\Omega}\left|\nabla v(x)\right|^2 \ud x}{\int_{\Omega}|v(x)|^2 \ud x}.
\end{align}
We can write the right hand side of \eqref{eqn:eigenvalue2} as
\begin{align*}
 \frac{\int_{\Omega}\chi(x_1)^2\frac{N_1^2}{(N_1-x_1)^2}\left|\psi'\left(x_2 N_1/(N_1-x_1)\right)\right|^2 \ud x}{\int_{\Omega}\chi(x_1)^2\left|\psi\left(x_2 N_1/(N_1-x_1)\right)\right|^2 \ud x} +  \frac{\int_{\Omega}\left|\pa_{x_1}v(x)\right|^2 \ud x}{\int_{\Omega}|v(x)|^2 \ud x},
\end{align*}
and on the support of $\chi(x_1)$ we have
\begin{align*}
\left|\frac{N_1}{N_1-x_1}-1 \right| \leq CN_1^{-2/3}.
\end{align*}
Therefore, since $\psi(x_2)$ is an eigenfunction on $\Omega(0)$ with eigenvalue $\mu_1^*$, we have
\begin{align} \label{eqn:eigenvalue3}
\la \leq \mu_1^* + CN_1^{-2/3} +   \frac{\int_{\Omega}\left|\pa_{x_1}v(x)\right|^2 \ud x}{\int_{\Omega}|v(x)|^2 \ud x}.
\end{align}
The $x_1$-derivative of $v$ is given by
\begin{align*}
\pa_{x_1}v(x_1,x_2) = \chi'(x_1)\psi\left(x_2 N_1/(N_1-x_1)\right) - \chi(x_1)\frac{x_2N_1}{(N_1-x_1)^2}\psi'\left(x_2 N_1/(N_1-x_1)\right).
\end{align*}
We have $|\chi'(x_1)| \leq CN_1^{-1/3}$, $|\psi'\left(x_2 N_1/(N_1-x_1)\right)|\leq C$, and $|x_2|\leq C$ on the support of $\psi$. Combining this with the estimate $N_1/(N_1-x_1)^2 \leq CN_1^{-1}$ on the support of $\chi(x_1)$, from \eqref{eqn:eigenvalue3} we obtain
\begin{align*}
\la \leq \mu_1^* + CN_1^{-2/3},
\end{align*}
as required. 
\end{proof}
We now prove the general case.
\begin{proof}{Proposition \ref{prop:eigenvalue}}
We first recall that the estimate in the proposition holds for $j=n$, and that the lower bound holds for all $j$. We will prove the upper bound by induction on $j$, using $j=n$ as the base case. Our inductive hypothesis is that there exists constants $C_j$ such that
\begin{align} \label{eqn:induction}
 \lambda \leq \mu_j^* + C_j N_j^{-2/3}
\end{align}
for $k+1 \leq j \leq n$, and we will prove that there exists a constant $C_k$ such that \eqref{eqn:induction} holds for $j=k$. Analogously to the two dimensional case, we will prove this estimate by using an appropriate test function in the variational formulation for $\lambda$. The minimal value $\mu_k^*$ is given by $\mu(Y_k)$ for some $Y_k\in\R^{k}$, and we let $\psi(X_{n-k}')$ be the $L^2(\Omega(Y_k))$-normalized first Dirichlet eigenfunction of the $(n-k)$-dimensional cross-section $\Omega(Y_k)$, and extended to be zero outside $\Omega(Y_k)$. (We recall that in our notation $X_{n-k}'= (x_{k+1},x_{k+2},\ldots,x_n)$.) Our test function will involve this eigenfunction, and we first use Proposition \ref{prop:derivative} to establish bounds on the components of the gradient of $\psi(X_{n-k}')$, under the inductive hypothesis.
\begin{lem} \label{lem:gradient}
Assuming that the estimate in \eqref{eqn:induction} holds for $j$ satisfying $k+1 \leq j \leq n$, there exists a constant $C$ $($depending on the constants $C_j$$)$ so that for each such $j$ in this range, 
\begin{align*}
\int_{\Omega(Y_k)}\left|\pa_{x_j}\psi(X_{n-k}')\right|^2 \ud X_{n-k}' \leq CN_{j}^{-2/3}\int_{\Omega(Y_k)}\left|\psi(X_{n-k}')\right|^2 \ud X_{n-k}'  = CN_j^{-2/3}.
\end{align*}
\end{lem}
\begin{proof}{Lemma \ref{lem:gradient}}
The eigenfunction $\psi(X_{n-k}')$ on $\Omega(Y_k)$ has eigenvalue $\mu_k^*$, and analogously to Definition \ref{def:W-eigenvalue}, for $k+1\leq j \leq n$, we define $\mu_{k,j}^*$ to be the minimum eigenvalue over all $(n-j)$-dimensional cross-sections of $\Omega(Y_k)$ in the $X_{n-j}'$ variables. %$x_{j+1},x_{j+2},\ldots,x_n$ variables. 
Since $\Omega(Y_k)\subset \Omega$, by the definitions of the minima $\mu_{k,j}^*$ and $\mu_j^*$ we automatically have
\begin{align*}
\mu_j^* \leq \mu_{k,j}^*.
\end{align*}
Combining this with the inductive hypothesis in \eqref{eqn:induction}, for each $k+1\leq j \leq n$ we obtain
\begin{align} \label{eqn:induction1}
\mu_k^* \leq \lambda \leq \mu_j^* + C_jN_j^{-2/3} \leq \mu_{k,j}^* + C_jN_j^{-2/3}.
\end{align}
Therefore, setting $W$ to be the $(n-k)$-dimensional convex domain $\Omega(Y_k)$, and using the notation from Definition \ref{def:W-eigenvalue} we have
\begin{align*}
\delta_i(W) = \lambda(W) - \mu_i^*(W) = \mu_k^* - \mu_{k,i+k}^* \leq C_{i+k}N_{i+k}^{-2/3}.
\end{align*} 
for $1 \leq i \leq n-k$. The gradient bounds in the statement of the lemma then immediately follow from Proposition \ref{prop:derivative}, using that $\psi(X_{n-k}')$ is $L^2(\Omega(Y_k))$-normalized. 
\end{proof}
We now define the test function that we will use to bound $\lambda$. We first translate the domain $\Omega$ in the $X_k$-variables so that the point $Y_k$ with $\mu(Y_k) = \mu_k^*$ is at the origin, which we denote by $0_k$. Then, using the above notation, $\psi(X_{n-k}')$ is the first Dirichlet eigenfunction of the $(n-k)$-dimensional cross-section $\Omega(0_k)$. By the properties of the John ellipsoid of $\Omega$, there exists a $k$-dimensional parallelepiped $P$ of dimensions comparable to  $N_1\times N_2\times\cdots \times N_k$ contained in the intersection of $\Omega$ with a $k$-dimensional plane $\{X_{n-k}' = \text{ constant}\}$. By translating $\Omega$ in the $X_{n-k}'$ variables we will assume that this $k$-dimensional plane is $\{X_{n-k}' = 0'_{n-k}\}$. Note that after this translation, there exists a constant $C$ such that
\begin{align} \label{eqn:projection}
\text{proj}_{j}(\Omega(0_k)) \subset \{ |x_{j}| \leq CN_j\}
\end{align}
for $k+1 \leq j \leq n$. Here $\text{proj}_{j}(\Omega(0_k))$ is the projection of $\Omega(0_k)$ onto the $x_j$-axis. Since $\Omega$ contains the above parallelepiped $P$, there exists a $(k-1)$-dimensional sphere contained in $\{X_{n-k}' = 0'_{n-k}\}$, centred at the origin $0_k$ in the $X_{k}$-variables, of radius $R_1$ with $R_1\sim N_1$, and with the following property: There exists a direction $\textbf{e}$ in the $X_k$-variables and number $\theta_k$, with $\theta_k\sim N_k/N_1$, such that the subset, $S_k$, of the sphere making an angle at most $\theta_k$ with $\textbf{e}$, is contained within $\Omega$. (Note that in the case of $k=1$, the sphere is $0$-dimensional, and the above reduces to the existence of a point in $\Omega$ at a distance comparable to $N_1$ from the $(n-1)$-dimensional cross-section $\Omega(0_1)$.)
\\
\\
We now let $\Gamma_k$ be the $k$-dimensional cone in the $X_k$-variables generated by the set $S_{k}$, with vertex at the origin $0_k$. This cone $\Gamma_k$ contains a $k$-dimensional cube of side length comparable to $N_k^{1/3}$, at a distance comparable to $N_1N_k^{-2/3}$ from the origin. We can therefore define a cut-off function $\chi(X_k)$ adapted to this cube (so that $\chi(X_k) = 1$ in the middle half of the cube, and $0$ outside the cube), with $|\nabla \chi(X_k)| \leq CN_k^{-1/3}$. Our test function is then
\begin{align} \label{eqn:test1}
w(x) = w(X_k,X_{n-k}') = \chi(X_k)\psi\left(X_{n-k}'R_1/(R_1 - r_k)\right).
\end{align}
Here $r_k = (x_1^2+x_2^2+\cdots+x_k^2)^{1/2}$ is the distance to the origin $0_k$ in the $X_k$-plane. Since $\Omega$ is convex, it contains the convex hull of the $(n-k)$-dimensional cross-section $\Omega(0_k)$ and the set $S_k$. Therefore, given $X_k\in S_k$, $s\in[0,1]$, the $(n-k)$-dimensional cross-section of $\Omega$ at $sX_k\in\Gamma_k$ contains the set 
\begin{align*}
\left(\frac{R_1-|sX_k|}{R_1}\right)\Omega(0_k)=(1-s)\Omega(0_k).
\end{align*}
Thus, the test function $w(x)$ vanishes outside of $\Omega$, and so can be used to obtain an upper bound on $\lambda$. We therefore have
\begin{align} \label{eqn:test2}
\la \leq \frac{\int_{\Omega}\left|\nabla_{X_k}w(x)\right|^2\ud x}{\int_{\Omega}\left|w(x)\right|^2\ud x} +  \frac{\int_{\Omega}\left|\nabla_{X_{n-k}'}w(x)\right|^2\ud x}{\int_{\Omega}\left|w(x)\right|^2\ud x},
\end{align}
and we deal with each term separately. We can write the second term in \eqref{eqn:test2} as
\begin{align} \label{eqn:test3}
\frac{\int_{\Omega}\frac{R_1^2}{(R_1-r_k)^2}|\chi(X_k)|^2\left|\left(\nabla_{X_{n-k}'}\psi\right)\left(X_{n-k}'R_1/(R_1 - r_k)\right)\right|^2\ud x}{\int_{\Omega}\left|\chi(X_k)\right|^2\left|\psi\left(X_{n-k}'R_1/(R_1 - r_k)\right)\right|^2\ud x},
\end{align}
and on the support of $\chi(X_k)$ we have
\begin{align} \label{eqn:test4}
\left|\frac{R_1}{R_1-r_k} -1\right| \leq CN_k^{-2/3}. 
\end{align}
Therefore, since $\psi(X_{n-k}')$ has eigenvalue $\mu_k^*$ on $\Omega(0)$, we can bound the quantity in \eqref{eqn:test3} by $\mu_k^* + CN_k^{-2/3}$. We now turn to the first term in \eqref{eqn:test2}. We can bound the magnitude of $\nabla_{X_k}w(x)$ by
\begin{align} \label{eqn:test5}
 \left|\left(\nabla\chi(X_k)\right) \psi\left(X_{n-k}'R_1/(R_1 - r_k)\right)\right| + \left|\chi(X_k)\frac{R_1}{(R_1-r_k)^2}X_{n-k}'\cdot \left(\nabla_{X_{n-k}'}\psi\right)\left(X_{n-k}'R_1/(R_1 - r_k)\right)\right|. 
\end{align}
Since $|\nabla\chi(X_k)|\leq CN_k^{-1/3}$, the contribution from the first term in \eqref{eqn:test5} leads to a contribution of size $CN_k^{-2/3}$ to \eqref{eqn:test2}. Using $|R_1-r_k|\geq cN_1$, together with the lengths of the projections of $\Omega(0)$ onto each axis from \eqref{eqn:projection}, we can bound the second term in \eqref{eqn:test5} by
\begin{align*}
CN_1^{-1}\sum_{j=k+1}^{n}N_j\left|\left(\pa_{x_{j}}\psi\right)\left(X_{n-k}'R_1/(R_1 - r_k)\right)\right|.
\end{align*}
Therefore, by Lemma \ref{lem:gradient}, we can bound the contribution to \eqref{eqn:test2} from the second term in \eqref{eqn:test5} by
\begin{align*}
CN_1^{-2}\sum_{j=k+1}^{n}N_j^{2}N_j^{-2/3} = CN_1^{-2}\sum_{j=k+1}^{n}N_j^{4/3}.
\end{align*}
Since $N_1\geq N_2\geq\cdots \geq N_n$, this can be bounded by $CN_1^{-2}N_{k+1}^{4/3} \leq CN_k^{-2/3}$. Putting everything together, we obtain
\begin{align*}
\la \leq \mu_k^*+ CN_k^{-2/3}.
\end{align*}
This is precisely the inductive step, and so completes the proof of the proposition. 
%\begin{enumerate}
%\item Translate in $X_k$-plane. Find $k$-dimensional plane with $N_j$ dimensions perpendicular to cross-section. Translate this in $X_{n-k}'$-plane (and note $N_j$-bounds on the location of the cross-section after this translation)
%\item Existence of a sphere of radius $R$ comparable to $N_1$ (in the $X_{k}$-plane), and a direction so that part of sphere with angles $\theta\leq \theta_k\sim N_k/N_1$  contained in $\Omega$. (Plus remark about special case of $k=1$, where just a point)
%\item Then define test function by first noting square in cone formed at relevant distance from tip in $X_k$-plane, then defining cut-off function and whole test function
%\item Remark why this is a admissible test function
%\item Obtain the bound by splitting gradient into two parts
%\end{enumerate}
%\begin{enumerate}
%\item Recall true for $j=n$ and write down inductive hypothesis (true for $k+1\leq j \leq n$), and then what we will show.
%\item Briefly describe test function and in particular which eigenfunction ($\mu_k^*$ and $\psi(X_{n-k}')$) it will involve.
%\item State a lemma involving the gradient bound satisfied by this eigenfunction under the inductive hypothesis.
%\item Prove this lemma using Proposition \ref{prop:derivative} and eigenvalue bounds inherited from the inductive hypothesis.
%\item Carefully describe the test function and write down the eigenvalue upper bound.
%\item Split the gradient up into two terms..one dealt with as in two dimensional case, and then the other dealt with using the lemma.
%\end{enumerate}
\end{proof}
\begin{rem}
Denoting $\la_m$ to be the $m$-th Dirichlet eigenvalue of $\Omega$, a small modification of the proof of Proposition \ref{prop:eigenvalue} ensures the existence of a constant $C_{m,n}$ such that
\begin{align} \label{eqn:rem-m}
\mu_j^* \leq \la_m \leq \mu_j^* + C_{m,n}N_j^{-2/3}.
\end{align}
The only change is that in place of $\chi(X_k)$, we require $m$ functions $\chi_{m}(X_k)$, with $|\nabla \chi_m(X_k)| \leq C_mN_k^{-1/3}$, chosen such that
\begin{align*}
w_m(x) = \chi_m(X_k)\psi(X_{n-k}'R_1/(R_1-r_k))
\end{align*}
are orthogonal. The estimate in \eqref{eqn:rem-m} in particular ensures that if $u_m$ is the corresponding $m$-th eigenfunction, then it also satisfies the derivative estimates in Theorem \ref{thm:Main1} with a constant $C_{m,n}$.
\end{rem}

\section{A lower bound on the $L^2(\Omega)$-norm of the eigenfunction} \label{sec:L2}

In this section we prove Theorem \ref{thm:Main} by combining the derivative estimates from Theorem \ref{thm:Main1} with the log concavity of the eigenfunction. Since $u$ is log concave, the superlevel set $\Omega_{1/2}$ is a convex subset of $\Omega$. In particular, we can associate a John ellipsoid $E_{1/2}$ to $\Omega_{1/2}$. Let $v_j$ be the unit directions along which the axes of $E_{1/2}$ lie, and let $M_{j}$ be the corresponding lengths of the axes. We also let $e_j$ be the unit directions along the cartesian coordinate axes. The first step is to show that $\Omega_{1/2}$ determines the $L^2(\Omega)$-norm of $u$.
\begin{lem} \label{lem:log}
There exist constants $C_1$, $c_1>0$ such that
\begin{align*}
c_1\prod_{j=1}^{n}M_j  \leq \int_{\Omega}|u(x)|^2 \ud x \leq C_1\prod_{j=1}^{n}M_j. 
\end{align*}
\end{lem}
\begin{proof}{Lemma \ref{lem:log}}
The lower bound follows immediately from the definitions of $M_j$. To obtain the upper bound we use the log concavity of $u$: The projection of the superlevel set $\Omega_{1/2}$ onto each $v_j$ axis is comparable to $M_j$. The function $\log(u)$ is concave and attains a maximum of $0$ in $\Omega$. Therefore, the projection of the sets
\begin{align*}
\Omega_{2^{-m}} = \{x\in\Omega:u(x) \geq 2^{-m}\} = \{x\in\Omega:\left|\log(u(x))\right| \leq m\left|\log(1/2)\right|\}
\end{align*}
onto each $v_j$ axis is at most a constant multiplied by $m M_j$. Summing this estimate over $m$ gives the desired upper bound in the lemma.
\end{proof}
We now reorder the directions $v_j$ to ensure that $M_1\geq M_2\geq\cdots \geq M_n$, and to prove Theorem \ref{thm:Main} we will obtain a lower bound on each $M_j$. Since $\Omega$ has inner radius comparable to $1$, the point where $u$ attains its maximum is at a distance at least $c>0$ from the boundary (see Theorem 1 in \cite{RS} in two dimensions, and Theorem 1.6 in \cite{GM} in higher dimensions). Therefore, by interior elliptic estimates, $M_n$ is certainly comparable to $1$. Given $k$, with $1\leq k \leq n-1$, let $w_k$ be a unit direction in $\R^n$ which lies in the projection of $\R^n$ onto the first $k$ coordinates. That is, $w_k$ is a linear combination of $e_j$ for $1\leq j \leq k$. We then consider the cross-sections of $\Omega$
\begin{align*}
\Omega_{w_k}(t) = \{x\in\Omega: x\cdot w_k = t\},
\end{align*} 
which as $t$ varies give the $(n-1)$-dimensional slices of $\Omega$ which are orthogonal to $w_k$.  For each $t$, we can consider the $L^2(\Omega_{w_k}(t))$-norm squared of $u$,
\begin{align} \label{eqn:slice1a}
\int_{\Omega_{w_k}(t)}|u(x)|^2 \ud \sigma_{n-1}(x;w_k),
\end{align}
where $\ud \sigma_{n-1}(x;w_k)$ is the flat $(n-1)$-dimensional surface measure on $\Omega_{w_k}(t)$. Suppose that the expression in \eqref{eqn:slice1a} is maximized when $t=t^*$, and set
\begin{align*}
B_k^* = \int_{\Omega_{w_k}(t^*)}|u(x)|^2 \ud \sigma_{n-1}(x;w_k).
\end{align*}
We can now use Theorem \ref{thm:Main1} to obtain a lower bound on the $L^2$-norm of $u$ in terms of $B_k^*$.
\begin{lem} \label{lem:slice}
There exists a constant $c_2>0$ such that for each $k$, $1\leq k \leq n-1$, and any such direction $w_k$,
\begin{align*}
\int_{\Omega}|u(x)|^2 \ud x  \geq c_2B_k^* N_k^{1/3}. 
\end{align*}
\end{lem}
\begin{proof}{Lemma \ref{lem:slice}}
Fix a point $x_{t^*}\in \Omega_{w_k}(t^*)$ and for each $s$ choose $x_s$ such that $(x_{t^*}-x_s)\cdot w_k = t^*-s$ and $|x_{t^*}-x_s| = |t^*-s|$. Then, extending $u$ by zero outside $\Omega$, for any $t$ we can write
\begin{align*}
u(x_t) = u(x_{t^*}) + \int_{t^*}^{t} \pa_{w_k}u(x_s) \ud s,
\end{align*}
where $\pa_{w_k}u$ is the directional derivative $w_k\cdot \nabla u$. This implies that
\begin{align*}
|u(x_t)|^2 & \geq \tfrac{1}{2}|u(x_{t^*})|^2 - \left(\int_{t^*}^{t}\pa_{w_k}u(x_s)\ud s \right)^2 \\
&\geq \tfrac{1}{2}|u(x_{t^*})|^2 - |t-t^*|\left|\int_{t^*}^{t}\left| \pa_{w_k}u(x_s)\right|^2 \ud s \right|.
\end{align*}
We now integrate over the $(n-1)$ variables orthogonal to $w_k$. Since $w_k$ lies in the projection of $\R^n$ onto the first $k$ coordinates, we can use Theorem \ref{thm:Main1} with $j\leq k$ to bound $\pa_{w_k}u$. We therefore have
\begin{align} \label{eqn:slice1}
\int_{\Omega_{w_k}(t)} |u(x)|^2 \ud \sigma_{n-1}(x;w_k) \geq \tfrac{1}{2}\int_{\Omega_{w_k}(t^*)}|u(x)|^2 \ud \sigma_{n-1}(x;w_k) -C|t-t^*|N_k^{-2/3} \int_{\Omega}|u(x)|^2 \ud x.
\end{align}
In particular, for
\begin{align*}
|t-t^*| \leq \tfrac{1}{4}C^{-1}N_k^{2/3} \left(\int_{\Omega}|u(x)|^2 \ud x\right)^{-1} \int_{\Omega_{w_k}(t^*)}|u(x)|^2 \ud \sigma_{n-1}(x;w_k),
\end{align*}
the estimate in \eqref{eqn:slice1} implies that
\begin{align*}
\int_{\Omega_{w_k}(t)} |u(x)|^2 \ud \sigma_{n-1}(x;w_k) \geq \tfrac{1}{4}\int_{\Omega_{w_k}(t^*)}|u(x)|^2 \ud \sigma_{n-1}(x;w_k)  = \tfrac{1}{4}B_k^*. 
\end{align*}
Therefore,
\begin{align*}
\int_{\Omega}|u(x)|^2 \ud x \geq  \tfrac{1}{16}C^{-1}N_k^{2/3} \left(\int_{\Omega}|u(x)|^2 \ud x\right)^{-1} \left(B_k^*\right)^2,
\end{align*}
and rearranging implies the estimate in the lemma. 
\end{proof}
The final step is to show that for each $k$, $1\leq k \leq n-1$, we can choose such a unit direction $w_k$ lying in $k$-dimensional space spanned by $e_1,e_2,\ldots,e_k$, such that
\begin{align} \label{eqn:Bk}
B_k^* \geq c_3 \prod_{j=1, j\neq k}^{n}M_j. 
\end{align} 
Inserting this in Lemma \ref{lem:slice} and using the upper bound in Lemma \ref{lem:log} implies that $M_k$ is bounded from below by a multiple of $N_k^{1/3}$. The lower bound in Lemma \ref{lem:log} then gives the estimate in Theorem \ref{thm:Main}. 
\\
\\
We are left to prove \eqref{eqn:Bk}, and we first consider $k=1$: Consider the $(n-1)$-dimensional cross-sections of $\Omega_{1/2}$ perpendicular to $w_1 = e_1$. Since $\Omega_{1/2}$ has volume comparable to $\prod_{j=1}^{n}M_j$ and diameter comparable to $M_1$, the volume of one of these cross-sections must be at least comparable to $\prod_{j=2}^{n}M_j$. In particular, this ensures that $B_1^*\geq \tfrac{1}{4}c\prod_{j=2}^{n}M_j$. 
\\
\\
For $k\geq2$, we first choose a unit direction $w_k$ in the intersection of the $k$-dimensional plane spanned by $e_1,e_2,\ldots,e_k$ and the $(n-k+1)$-dimensional plane spanned by $v_k,v_{k+1},\ldots,v_n$. Taking the $(n-1)$-dimensional cross-sections of $\Omega_{1/2}$ perpendicular to $w_k$, the volume of one of these cross-sections must be at least $c\prod_{j=1, j\neq k}^{n}M_j$. To see this, we first note that there is a $(n-k+1)$-dimensional cross-section of $\Omega_{1/2}$ which is perpendicular to $v_1,v_2,\ldots,v_{k-1}$ and contains a $(n-k+1)$-dimensional ellipsoid $E$ with axes of lengths $M_k,M_{k+1},\ldots,M_n$. In particular, the volume of one of the $(n-k)$-dimensional cross-sections of $E$ which is perpendicular to $w_k$ must be at least $c\prod_{j=k+1}^{n}M_j$. But $w_k$ is also perpendicular to $v_1,v_2,\ldots,v_{k-1}$, and the projection of $\Omega_{1/2}$ onto the $v_j$-direction is comparable to $M_j$. Therefore, there exists a $(n-k)+(k-1) = (n-1)$-dimensional cross-section of $\Omega_{1/2}$ perpendicular to  $w_k$ of volume at least $c\left(\prod_{j=1}^{k-1}M_j\right)\left(\prod_{j=k+1}^{n}M_j\right)$. This ensures that $B_k^*\geq \tfrac{1}{4}c\prod_{j=1, j\neq k}^{n}M_j$, and \eqref{eqn:Bk} holds.

%\section{A discussion of the sharpness of these estimates}
\section{The two-dimensional case} \label{sec:2d}

Theorem \ref{thm:Main} provides a lower bound on the $L^2(\Omega)$-norm of $u$. In two dimensions, Jerison and Grieser have given a precise characterization of the shape of $u$ in terms of the geometry of $\Omega$. To state this, we first rotate so that the projection of the planar domain onto the $x_2$-axis is the smallest and dilate so that this projection is of length $1$. Then, we can write $\Omega$ as
\begin{align*}
\Omega = \{(x_1,x_2)\in \R^2: a\leq x_1\leq b, f_1(x_1)\leq x_2 \leq f_2(x_1)\}.
\end{align*}
Here $b-a$ is comparable to $N_1$, $f_1$, $f_2$ are convex, concave functions respectively, and $0\leq h(x) = f_2(x_1)-f_1(x_1)$ is a concave function, attaining a maximum of $1$. 
\begin{defn}[\cite{Je1}] \label{def:L}
Define $L$ to be the largest value such that $1-L^{-2} \leq h(x_1) \leq 1$ on an interval $I$ of length $L$. 
\end{defn}
Since $h(x_1)$ is concave, the value of $L$ satisfies 
%\begin{align*}
$cN_1^{1/3} \leq L \leq CN_1$, 
%\end{align*}
and $L\sim N_1$, $L\sim N_1^{1/3}$ is attained when $\Omega$ is a rectangle, circular sector respectively. Any intermediate value of $L$ can be obtained by, for example, forming the trapezoid of a rectangle of diameter $L$ attached to a right angled triangle. In \cite{Je1}, \cite{GJ2}, \cite{GJ1}, Grieser and Jerison obtain estimates on the first and second Dirichlet eigenfunction in terms of this length scale $L$. Their approach is to perform an approximate separation of variables in $\Omega$. Since the cross-section of $\Omega$ at $x_1$ has eigenvalue $\tfrac{\pi^2}{h(x_1)^2}$, a separation of variables leads to the ordinary differential operator
\begin{align*}
\mathcal{L} = -\frac{d^2}{dx_1^2} + \frac{\pi^2}{h(x_1)^2}
\end{align*}
on the interval $[a,b]$. Grieser and Jerison approximate $\la$ and $u$ in terms of the first eigenvalue and eigenfunction of $\mathcal{L}$, and the approximation becomes stronger as the diameter of $\Omega$ increases. As a consequence of their work, the following $L^2(\Omega)$ bound holds in this planar case.
\begin{thm}[Grieser-Jerison, \cite{GJ1}] \label{thm:L2-2d}
There exists an absolute constant $C$ such that the superlevel set $\{u>\tfrac{1}{2}\max_{\Omega}u\}$ has diameter bounded between $C^{-1}L$ and $CL$, and 
\begin{align*}
C^{-1}L^{1/2}\norm{u}_{L^{\infty}(\Omega)}\leq \norm{u}_{L^2(\Omega)} \leq CL^{1/2}\norm{u}_{L^{\infty}(\Omega)}.
\end{align*}
\end{thm}
Using the definition of $L$ from Definition \ref{def:L} to compare the estimate in Theorem \ref{thm:L2-2d} with the lower bound in Theorem \ref{thm:Main} in two dimensions, we note the following. When $L$ is comparable to $N_1^{1/3}$, such as for a circular sector or right angled triangle,  the bounds in the two theorems agree and in particular the lower bound in Theorem \ref{thm:Main} is sharp. However, for $L\gg N_1^{1/3}$ Theorem \ref{thm:L2-2d} says that the eigenfunction $u$ has spread out by more than $N_1^{1/3}$ in the $x_1$-direction and so the $L^2(\Omega)$-norm of $u$ is larger than that given in Theorem \ref{thm:Main}.

In higher dimensions, we can begin an analogous discussion. Consider the thin sector in $\R^n$ of the form
\begin{align*}
\{(r,\theta):0<r<N_1,\theta\in D^{n-1}\},
\end{align*}
where $D^{n-1}$ is a geodesic disc of radius $1$ in $S^{n-1}$. As shown in \cite{vdB}, for this domain, the lower bound given in Theorem \ref{thm:Main} is sharp. If the domain $\Omega$ is instead a parallelepiped, then the superlevel set $\{u>\tfrac{1}{2}\max_{\Omega}u\}$ takes up a uniform portion of the whole domain. For a parallelepiped, this leads to the estimate
\begin{align*}
\norm{u}_{L^2(\Omega)} \sim \text{ Volume}(\Omega)^{1/2}\norm{u}_{L^{\infty}(\Omega)} \sim \prod_{j=1}^{n}N_j^{1/2} \norm{u}_{L^{\infty}(\Omega)}.
\end{align*}
Therefore, in dimensions higher than two it is natural to ask whether one can define analogous length scales to that of $L$ from Definition \ref{def:L} which govern the shape of the first eigenfunction.
\begin{ques} \label{ques1}
 Fix $c$, with $0<c<1$. Can we use the geometry of $\Omega$ to determine $n$ length scales $M_1\geq M_2\geq\cdots \geq M_{n}$, and $n$ directions $v_1,v_2,\ldots,v_{n}$ in $\R^{n}$ such that the John ellipsoid of
\begin{align*}
\{x\in\Omega:u(x)>c\max_{\Omega}u\}
\end{align*}
has axes along the directions $v_j$ and of lengths comparable to $M_j$?
\end{ques}
This question is open in any dimension higher than two. Let us normalize $\Omega\subset \R^n$ so that it has inner radius equal to $1$, and its projection onto the $x_n$-axis is of length comparable to $1$. Then, we can certainly choose $v_{n}$ to point in the $x_n$-direction and take $M_{n}\sim1$. The question is then to determine the remaining $n-1$ length scales and orientation. The results of this paper show that the lengths $M_j$ must satisfy the lower bound $M_j\geq c N_j^{1/3}$. In \cite{B1}, another preliminary step towards answering this question has been carried out: Consider the operator 
\begin{align} \label{eqn:potential}
-\Delta_{x_1,x_2} +  \frac{\pi^2}{h(x_1,x_2)^2}, %\frac{d^2}{dx_1^2} -\frac{d^2}{dx_2^2} + \frac{\pi^2}{h(x_1,x_2)^2}
\end{align}
with Dirichlet boundary conditions on a two dimensional convex domain $D$. Here $h(x_1,x_2)$ is a concave function on $D$, attaining a minimum of $1$. The first eigenfunction of this operator still has convex superlevel sets and in \cite{B1}, length scales $L_1$, $L_2$ and an orientation of the domain $D$ are found in terms of $D$ and $h$, which govern the intermediate level sets of this first eigenfunction. In particular, the $L^2(D)$-norm is comparable to $L_1^{1/2}L_2^{1/2}$ multiplied by the $L^{\infty}(D)$-norm of the eigenfunction. 

The operator in \eqref{eqn:potential} can be used to make progress of answering the question in the three dimensional case. For three dimensional domains of the form
\begin{align*}
\Omega = \{(x_1,x_2,x_3)\in \R^3: (x_1,x_2)\in D, 0\leq x_3 \leq h(x_1,x_2)\},
\end{align*}
an approximate separation of variables into $(x_1,x_2)$ and $x_3$-variables leads to the operator in \eqref{eqn:potential}. It is shown in \cite{B2} that when $L_1$ and $L_2$ are sufficiently close in size ($L_1 \leq L_2^{3/2-}$), this separation of variables provides a good approximation to the first eigenfunction of $\Omega$. In particular, referring back to Question \ref{ques1}, in this case we can set $M_1 = L_1$, $M_2 = L_2$, $M_3 = 1$, and the orientation of $D$ also governs the behaviour of the first Dirichlet eigenfunction of $\Omega$. To make further progress towards fully answering Question \ref{ques1}, even in the three dimensional case, a key step is to determine the orientation of the superlevel sets of $u$, as in general it will not be the same as that of $\Omega$ itself. Especially as the dimension of $\Omega$ increases, it is unclear how to determine this orientation.
 
% 
%\begin{enumerate}
%
%\item  Introductory sentence
%\item 2d normalization of the domain and put $L$ as a definition
%\item Possible range of $L$ and extremal examples
%\item Sentence description of what GJ do with it (and separation of variables mindset..mention ODE)
%\item Precise $L^2(\Omega)$-estimate and in particular when Theorem \ref{thm:Main} is and isn't sharp in two dimensions, in terms of intermediate level set
%\item Examples in higher dimensions where Theorem \ref{thm:Main} is and isn't sharp (boxes compared to cones)
%\item Pose as a question: $n-1$ quantities (after normalizing inner radius) governing shape of the first eigenfunction, in terms of shape of intermediate level set
%\item Some sort of brief description of my 3d work as a first step towards question: Write out the 2d plus potential operator, with $\pi^2/h^2$, $h$ concave, min $h=1$. State result about first eigenfunction of this operator. Then, used to handle certain 3d case with height function $h$.
%\item Remark about lower bound on quantities coming from Theorem \ref{thm:Main}. Plus talk about orientation being a key step, and not obvious how the level set lines up in higher dimensions..no reason same as domain itself.
%
%\end{enumerate}

\Addresses


\begin{thebibliography}{12}

\bibitem{ADJMF} D.~Arnold, G.~David, D.~Jerison, S.~Mayboroda, and M.~Filoche, \emph{The effective confining potential of quantum states in disordered media}, Phys. Rev. Lett. 116 (2016), Article Number: 0566

\bibitem{B1} T.~Beck, \emph{The shape of the level sets of the first eigenfunction of a class of two dimensional Schr\"odinger operators},  Trans. Amer. Math. Soc. 370 (2018), 3197--3244.%Trans. Amer. Math. Soc. (to appear). \\
%\url{http://arxiv.org/abs/1411.7353}

\bibitem{B2} T.~Beck, \emph{Level set shape for ground state eigenfunctions on convex domains}, PhD thesis.% at Princeton University.

\bibitem{BL} H. J. Brascamp and E. H. Lieb, \emph{On extensions of the Brunn-Minkowski and Prekopa Leindler theorems, including inequalities for log concave functions, and with an application to the diffusion equation}, J. Funct. Anal. 22 (1976), 366--389.

\bibitem{Ch1} G.~Chiti, \emph{A reverse H\"older inequality for the eigenfunctions of linear second order elliptic operators}, J. Applied Mathematics and Physics 33 (1982), 143--148.

\bibitem{Ch2} G.~Chiti, \emph{An isoperimetric inequality for the eigenfunctions of linear second order elliptic operators}, Boll. Un. Mat. Ital. A 1 (1982), 145--151.

\bibitem{FM}  M.~Filoche and S.~Mayboroda, \emph{Universal mechanism for Anderson and weak localization}, Proc. Natl. Acad. Sci. USA 109 (2012), 14761--14766.

\bibitem{GM} B. Georgiev and M. Mukherjee, \emph{Nodal Geometry, Heat Diffusion and Brownian Motion}, Anal. PDE 11 no. 1 (2018), 133--148. 

\bibitem{GMS} B.~Georgiev, M.~Mukherjee, and S.~Steinerberger, \emph{A Spectral Gap Estimate and Applications}, Potential Analysis, 49, (2018) 635--645. 

\bibitem{GJ2} D. Grieser and D. Jerison, \emph{Asymptotics of the first nodal line of a convex domain}, Invent. Math. 125 no. 2 (1996), 197--219.

\bibitem{GJ1} D. Grieser and D. Jerison, \emph{The size of the first eigenfunction of a convex planar domain}, J. Amer. Math. Soc. 11 no. 1 (1998), 41--72.

\bibitem{Je1} D. Jerison, \emph{The diameter of the first nodal line of a convex domain}, Ann. of Math. 141 (1995), 1--33.

\bibitem{Jo} F. John, \emph{Extremum problems with inequalities as subsidiary conditions}, Studies and Essays Presented to R. Courant on his 60th Birthday. January 8 (1948), pp.187--204.

\bibitem{Kr1} P.~Kr\"oger, \emph{On the ground state eigenfunction of a convex domain in Euclidean space}, Potential Analysis 5 (1996), 103--108.

\bibitem{RS} M. Rachh and S. Steinerberger, \emph{On the location of maxima of solutions of Schr\"odingerÕs equation}, Comm. Pure. Appl. Math. 71, (2018) 525--537.

\bibitem{SS1} S.~Steinerberger, \emph{Localization of Quantum States and Landscape Functions}, Proc. Amer. Math. Soc., 145 (2017), 2895--2907. 

\bibitem{vdB} M.~van den Berg, \emph{On the $L^{\infty}$-Norm of the First Eigenfunction of the Dirichlet Laplacian}, Potential Analysis 13 (2000), 361--366.

\end{thebibliography}
\end{document}